\documentclass[a4paper, 11pt]{article}
\usepackage[utf8]{inputenc}
\usepackage{amsmath, amssymb, amsthm, authblk, bm, lipsum, url}
\usepackage[textwidth = 14cm]{geometry}
\usepackage[hidelinks, colorlinks = true, allcolors=blue]{hyperref}
\usepackage[square]{natbib}
\DeclareMathOperator{\vol}{vol}

\DeclareMathOperator{\lin}{lin}

\DeclareMathOperator{\conv}{conv}

\DeclareMathOperator{\pos}{pos}

\DeclareMathOperator{\intt}{int}

\DeclareMathOperator{\parr}{par}

\newcommand{\R}{\mathbb{R}}
\newcommand{\Z}{\mathbb{Z}}
\newcommand{\N}{\mathbb{N}}

\newcommand{\bx}{\bm{x}}

\newcommand{\bh}{\bm{h}}
\newcommand{\bB}{\bm{B}}

\newcommand{\br}{\bm{r}}

\newcommand{\be}{\bm{e}}
\newcommand{\bR}{\bm{R}}

\newcommand{\by}{\bm{y}}
\newcommand{\bz}{\bm{z}}

\newcommand{\bb}{\bm{b}}

\newcommand{\bv}{\bm{v}}

\newcommand{\bw}{\bm{w}}

\newcommand{\btau}{\bm{\tau}}
\newcommand{\bmu}{\bm{\mu}}
\newcommand{\blambda}{\bm{\lambda}}

\newtheorem{theorem}{Theorem}
\newtheorem{theoreml}{Theorem}

\newtheorem{proposition}[theorem]{Proposition}

\newtheorem{lemma}[theorem]{Lemma}
\newtheorem{corollary}[theorem]{Corollary}

\numberwithin{theorem}{section}

\newcommand{\Cr}{\operatorname{ICR}}

\newcommand\blfootnote[1]{%
	\begingroup
	\renewcommand\thefootnote{}\footnote{#1}%
	\addtocounter{footnote}{-1}%
	\endgroup
}

\title{Integer Carath\'{e}odory results with bounded multiplicity}

\author{Stefan Kuhlmann\thanks{Department of Mathematics, Institute for Operations Research, ETH Zürich, Switzerland, stefan.kuhlmann@math.ethz.ch}}

\date{}

\begin{document}
	\maketitle
	\noindent\textbf{Abstract.} The integer Carath\'{e}odory rank of a pointed rational cone $C$ is the smallest number $k$ such that every integer vector contained in $C$ is an integral non-negative combination of at most $k$ Hilbert basis elements. 
	We investigate the integer Carath\'{e}odory rank of simplicial cones with respect to their multiplicity, i.e., the determinant of the integral generators of the cone. 
	One of the main results states that simplicial cones with multiplicity bounded by five have the integral Carath\'eodory property, that is, the integer Carath\'eodory rank equals the dimension. Furthermore, we give a novel upper bound on the integer Carath\'eodory rank that depends on the dimension and the multiplicity. This bound improves upon the best known upper bound on the integer Carath\'eodory rank if the dimension exceeds the multiplicity. At last, we present special cones that have the integral Carath\'eodory property such as certain dual cones of Gorenstein cones. 

	\section{Introduction}
	Given a full-dimensional pointed rational cone $C\subseteq \R^n$, every vector contained in $C$ is a non-negative combination of the \textit{generators of $C$}, i.e., vectors that lie on extreme rays of $C$. A classical result by Carath\'{e}odory states that each $\bx\in C$ is a non-negative combination of at most $n$ generators. In this manner, an integral analogue of Carath\'{e}odory's theorem is to ask whether every integral vector $\bz\in C\cap\Z^n$ can be expressed as the non-negative integral combination of at most $n$ elements from the unique minimal generating set of $C\cap\Z^n$, the \textit{Hilbert basis of $C$}; see Section \ref{ss_notation_hilbert} for more details concerning Hilbert bases. 
	
	This was proven to be false \citep{brunsgubehenkcounterexampleintcara99}. The authors present cones $C_n$ that require at least $\left\lfloor \frac{7}{6}n\right\rfloor$ elements from the Hilbert basis to express certain integral vectors in $C_n$ for each $n\geq 1$. 
	This example raises the question of what the correct upper bound on the smallest number $k$ is such that every integer vector in $C$ is an integral non-negative combination of at most $k$ Hilbert basis elements. 
	We write $\Cr(C)$ for this number $k$ and refer to it as the \textit{integer Carath\'{e}odory rank of $C$}; see Section \ref{ss_notation_hilbert} for a precise definition. 
	The best known upper bound on $\Cr(C)$ is due to Seb\H{o}:
	\begin{theoreml}[{\cite[Theorem 2.1]{sebohilbertbasisdreidim90}}]
		\label{thm_sebo_icr}
		Let $C\subseteq\R^n$ be a full-dimensional pointed rational cone. Then 
		\begin{align*}
			\Cr(C)\leq 2n-2.
		\end{align*}
	\end{theoreml}
	The lower bound in \citep{brunsgubehenkcounterexampleintcara99} and Seb\H{o}'s upper bound remain the best known bounds since decades. It is open whether both bounds can be improved. Another challenge concerning the integer Carath\'{e}odory rank is to classify and identify cones that admit $\Cr(C) = n$, the minimal possible integer Carath\'{e}odory rank for full-dimensional cones. Throughout this paper, we refer to a cone $C$ with $\Cr(C) = n$ that $C$ has the \textit{integral Carath\'{e}odory property} and abbreviate this by (ICP). One of the important results concerning the (ICP) is again due to Seb\H{o}:
	\begin{theoreml}[{\cite[Theorem 2.2]{sebohilbertbasisdreidim90}}]
		\label{thm_sebo_three}
		Let $C\subseteq\R^n$ be a full-dimensional pointed rational cone with $n\leq 3$. Then 
		\begin{align*}
			\Cr(C)= n.
		\end{align*}
	\end{theoreml}
	From the counterexamples in \citep{brunsgubehenkcounterexampleintcara99} it is known that cones with $n\geq 6$ do not have the (ICP), in general. It is open whether cones in dimension $n\in\lbrace 4,5\rbrace$ admit the (ICP). 
	To prove both theorems, Theorem \ref{thm_sebo_icr} and Theorem \ref{thm_sebo_three}, Seb\H{o} relied on simplicial cones, i.e., full-dimensional cones that have precisely $n$ extreme rays; see Section \ref{ss_notation_hilbert} for a more general definition. He employed the following high-level strategy:
	\begin{enumerate}
		\item Reduce from the general cone $C$ to a special simplicial subcone $C'\subseteq C$.
		\item Exploit the simplicial structure of $C'$.
	\end{enumerate}
	In fact, this strategy is commonly utilized when proving results concerning $\Cr(C)$ or other variants of this problem; cf., for instance, with \cite[Chapter 3]{brunsgubeladzepolyringsktheory2009}. This leads to the following open intriguing question:
	\begin{center}
		 \textit{Does every simplicial cone have the (ICP)?}
	\end{center}
	An affirmative answer has far-reaching consequences: It implies that the zonotope spanned by the integral generators of a cone contains a set of integral vectors, not necessarily the Hilbert basis, that admits the (ICP). 
	
	We contribute to the question of every simplicial cone having the (ICP) by investigating simplicial cones with respect to their multiplicity. In the following, we make precise what we mean by multiplicity in the full-dimensional case; see Section \ref{ss_notation_hilbert} for a more general definition. Let $\br^1,\ldots,\br^n\in\Z^n$ be linearly independent vectors. Then $\pos\lbrace \br^1,\ldots,\br^n\rbrace$ is a full-dimensional simplicial cone, where $\pos X$ denotes the set of all finite non-negative combinations of elements in $X\subseteq\R^n$. The \textit{multiplicity of $\pos\lbrace \br^1,\ldots,\br^n\rbrace$} is given by $\left|\det(\br^1,\ldots,\br^n)\right|$. This quantity has various interpretations, for instance, it is known that $\left|\det(\br^1,\ldots,\br^n)\right|$ coincides with the number of integer vectors contained in the half-open parallelepiped spanned by $\br^1,\ldots,\br^n$; see, e.g., \cite[Lemma 2]{sebohilbertbasisdreidim90} for a proof.
	
	One of our main results states that simplicial cones with small multiplicity have the (ICP).
	\begin{theorem}
		\label{thm_simplicial_icp}
		Let $\br^1,\ldots,\br^n\in\Z^{n}$ satisfy $1\leq\left|\det (\br^1,\ldots,\br^n)\right|\leq 5$. Then 
		\begin{align*}
			\Cr(\pos\lbrace \br^1,\ldots,\br^n\rbrace)=n.
		\end{align*}
	\end{theorem}
	We also provide a novel bound on the integer Carath\'{e}odory rank that improves upon Seb\H{o}'s bound, Theorem \ref{thm_sebo_icr}, if the multiplicity is bounded by the dimension, that is, $\left|\det(\br^1,\ldots,\br^n)\right|\leq n$.
	\begin{theorem}
		\label{thm_simplicial_icr}
		Let $\br^1,\ldots,\br^n\in\Z^{n}$ satisfy $6\leq\left|\det (\br^1,\ldots,\br^n)\right|\leq n$. Then
		\begin{align*}
			\Cr(\pos\lbrace\br^1,\ldots,\br^n\rbrace)\leq n + \left| \det (\br^1,\ldots,\br^n)\right| - 3.
		\end{align*}
	\end{theorem}
	We allude that these statements resemble recent results on the integer Carath\'{e}o\-dory rank and the (ICP) when the simplicial cone is given as a polyhedral representation \cite[Theorem 4]{aliev2024new}. Although these results seem to be related, their proofs differ significantly. Also, we remark that it will become evident from the proofs that both results remain valid for $k$-dimensional simplicial cones in $\R^n$ with the respective notion of multiplicity; see Section \ref{ss_notation_hilbert} for an introduction of these concepts. 

	An immediate follow-up question is to ask whether these bounds generalize to non-simplicial cones. We emphasize that this is not the case for Theorem \ref{thm_simplicial_icp}: A natural generalization of the multiplicity for a full-dimensional non-simplicial cone $\pos\lbrace \br^1,\ldots,\br^t\rbrace$ given by $\br^1,\ldots,\br^t\in\Z^n$ is to define $\delta(\br^1,\ldots,\br^t)$ as the largest multiplicity of any simplicial subcone formed by $\br^1,\ldots,\br^t$, that is,
	\begin{align*}
		\delta(\br^1,\ldots,\br^t) := \max \left\lbrace \left|\det(\br^{i_1},\ldots,\br^{i_n})\right| : i_1,\ldots,i_n\in\lbrack t\rbrack\right\rbrace.
	\end{align*}
	If we assume that $\delta(\br^1,\ldots,\br^t) = 1$, then the cone $\pos\lbrace \br^1,\ldots,\br^t\rbrace$ satisfies a strong property: It is possible to construct a unimodular triangulation of $\pos\lbrace \br^1,\ldots,\br^t\rbrace$ using the vectors in $\lbrace \br^1,\ldots,\br^t\rbrace$. This implies that cones with $\delta(\br^1,\ldots,\br^t) = 1$ have the (ICP); see Section~\ref{ss_notation_hilbert} for a definition and discussion of unimodular triangulations. 
	However, in contrast to Theorem \ref{thm_simplicial_icp}, one example given in \citep{brunsgubehenkcounterexampleintcara99} has a representation that only uses $0/1$-vectors as generators. Those vectors collectively satisfy $\delta(\br^1,\ldots,\br^t) = 3$. So already cones with $\delta(\br^1,\ldots,\br^t) = 3$ do not have the (ICP). This demonstrates that the simplicial case contains more structure. 
	We emphasize that it is open whether cones with $\delta(\br^1,\ldots,\br^t) = 2$ have the (ICP).
	
	The key technique when proving Theorems \ref{thm_simplicial_icp} and \ref{thm_simplicial_icr} is a projection argument, Lemma \ref{lemma_bounded_complexity}, which we devote an extra section, Section \ref{sec_projective_lemma}. Afterwards, we prove our main results in Section \ref{section_main_proofs}. As an application of Lemma \ref{lemma_bounded_complexity}, one can collect various well-known cones that have the (ICP); see Corollary \ref{cor_three_coset} and Section \ref{sec_special_cones}. These include among others special instances of dual cones of simplicial Gorenstein cones.
	
	\subsection{Related work}
	Despite the existence of counterexamples to the question whether every cone has the (ICP), it remains an active research area to bound the integer Carath\'{e}odory rank and identify cones with minimal integer Carath\'{e}odory rank. In \citep{gijswijtipdicp2012}, the authors show that cones that arise from polyhedra with additional properties, e.g., polyhedra defined by a totally unimodular constraint matrix or (poly)ma\-troid base polytopes have the (ICP). 
	The latter improved upon an earlier result, which states that $\Cr(C)\leq  n + r(M) - 1$, where $C$ is given by the incidence vectors of bases of a matroid $M$ and $r(M)$ denotes the rank of $M$ \citep{pinamatroidcararank2003}. 
	More recent bounds on the integer Carath\'eodory rank and novel instances that admit the (ICP) are given in \citep{aliev2024new}. For those results, the authors assume that the cone is represented as the intersection of half-spaces.
	
	Seb\H{o} introduced stronger versions of the (ICP) \citep{sebohilbertbasisdreidim90}. He asked whether the Hilbert basis of a given cones gives a unimodular cover of the cone or, even stronger, a unimodular triangulation; see Section~\ref{ss_notation_hilbert} for a definition and discussion of unimodular covers and triangulations. Both questions turned out to be false since they would imply the (ICP). Nevertheless, this led to the search of quantitative upper bounds on the size of a set that admits a unimodular cover or triangulation; see \citep[Chapter 3B and 3C]{brunsgubeladzepolyringsktheory2009} or \citep{brunsthadenunimodtriangbound2017, thadenunimodtriangbound2021} for some more recent work. This work relies on carefully analyzing the simplicial case. 
	
	Bruns and Gubeladze introduced a notion of an asymptotic integer Carath\'{e}odory rank. They proved an upper bound of $2n - 3$ on this asymptotic version of the integer Carath\'{e}odory rank \citep{brunsgubeladzenormalpoly1999}. Recently, this was improved to $\lfloor\frac{3n}{2}\rfloor$ in \citep[]{aliev2024new}. This is the first improvement on any version of the integer Carath\'{e}odory rank that does not have the factor $2$ in front of $n$. In \citep{gubeladzesurveynormal2023}, the author conjectured that the $2$ in front of the $n$ in Theorem \ref{thm_sebo_icr} can not be improved for the integer Carath\'eodory rank of so-called normal polytopes. 
	
	Instead of analyzing the integer Carath\'{e}odory rank with respect to the Hilbert basis of a cone, it is also of interest to study the integer Carath\'{e}odory rank for general generating sets. This relates to studying the sparsity of feasible non-zero integral vectors to systems of linear Diophantine equations. 
	A first contribution to this is due to \citep{eisenbrandshmonincaratheodorybounds06}. Over the last years, more results on this topic and the related question of finding sparse integer optimal solutions were obtained; see \citep{alievdeloesparelindio2017,alideloeisoerweissupportint2018,alievAverkovLoeraOertel21,dubeyliu2023shortproofsparse}, to name a few.
	
	\section{Notation and definitions}\label{sec_notation_def}
	In this section, we introduce additional notation and definitions that we use throughout the paper. In particular, we define the notion of Hilbert bases, integer Carath\'{e}o\-dory rank, the multiplicity, unimodular covers, and unimodular triangulations. Moreover, we introduce lattices and discuss a projected version of the integer Carath\'{e}odory rank. 
	We abbreviate $\lbrack m\rbrack :=\lbrace 1,\ldots,m\rbrace$. By $\intt X$ we denote the interior of a set $X$ and $\lin X$ is the linear hull of $X$. The standard unit vectors of $\R^n$ are denoted by $\be^1,\ldots,\be^n$. Let $\br^1,\ldots,\br^k\in\Z^n$. We abbreviate the linear hull $\lin\lbrace \br^1,\ldots,\br^k\rbrace$ by $\lin \bR$ for $\bR=(\br^1,\ldots,\br^k)\in\Z^{n\times k}$. Similarly, we write $\pos\bR$ for $\pos\lbrace \br^1,\ldots,\br^k\rbrace$.
	
	\subsection{Hilbert bases, integer Carath\'{e}odory rank, multiplicity, and unimodular covers/triangulations}\label{ss_notation_hilbert}
	A cone $C$ is pointed if $\bm{0}$ is the vertex of $C$. The \textit{Hilbert basis} of a pointed rational cone $C$ with respect to $\Z^n$ is the minimal set of integral vectors in $C$ such that every element in $C\cap \Z^n$ is a non-negative integral combination of the elements in the set. We denote the Hilbert basis of $C$ by $\mathcal{H}(C)$. The integer vectors in the Hilbert basis are called \textit{Hilbert basis elements}. It is known that the Hilbert basis is finite and unique for each pointed rational cone $C$; see \citep{vandercorputhilberbasisunique1931}. Moreover, a Hilbert basis element can not be written as the sum of two non-zero integral vectors in $C\cap\Z^n$. 
	The \emph{integer Carath\'{e}odory rank of $C$} is defined as
	\begin{align*}
		\Cr(C) := \max_{\bz\in C\cap\Z^n}\min\left\lbrace k  : \bz = \sum_{i=1}^k\lambda_i\bh^i \text{ with }\lambda_i\in\N\text{, }\bh^i\in\mathcal{H}(C)\text{ for all }i\in\lbrack k\rbrack\right\rbrace.
	\end{align*}
	Let $\br^1,\ldots,\br^k\in\Z^n$ be linearly independent and $\bR = (\br^1,\ldots,\br^k)$. Then the $k$-dimensional cone $\pos\bR$ is referred to as $\emph{simplicial}$. We define by
	\begin{align*}
		\parr\lbrace\br^1,\ldots,\br^k\rbrace := \left\lbrace \bx\in\R^n : \sum_{i=1}^k\lambda_i \br^i, \lambda_i\in \lbrack 0,1) \text{ for all }i\in\lbrack k\rbrack\right\rbrace\cap \Z^n
	\end{align*}
	the set of integer vectors in the $k$-dimensional half-open parallelepiped spanned by $\br^1,\ldots,\br^k$. According to the previously introduced notation, we abbreviate $\parr \lbrace\br^1,\ldots,\br^k\rbrace$ by $\parr\bR$.  
	The \emph{multiplicity of $\pos\lbrace \br^1,\ldots,\br^k\rbrace$} is given by the number of integer vectors contained in the half-open parallelepiped spanned by $\br^1,\ldots,\br^k$, that is, $\left| \parr\lbrace\br^1,\ldots,\br^k\rbrace\right|$. We write $\Delta(\br^1,\ldots,\br^k)$ for the multiplicity. Observe that $\Delta(\br^1,\ldots,\br^n) = \left|\det(\br^1,\ldots,\br^n)\right|$ when $k = n$; see, e.g., \cite[Lemma 2]{sebohilbertbasisdreidim90} for a proof. If $\Delta(\br^1,\ldots,\br^k) = 1$, we call $\pos\bR$ \emph{unimodular}. 
	Given a finite set $S\subseteq \pos\bR \cap \Z^n$, then a \emph{cover of $\pos\bR$ using $S$} is given by a collection of simplicial cones $\mathcal{C}$ such that $\pos\bR\subseteq \bigcup_{C\in\mathcal{C}} C$ and for each $C\in\mathcal{C}$ the generators can be chosen to be elements in $S$. If in addition to this each $C\in \mathcal{C}$ is unimodular, we refer to $\mathcal{C}$ as a \emph{unimodular cover of $\pos\bR$ using $S$}. A \emph{triangulation of $\pos\bR$ using $S$} is given by a collection of simplicial cones $\mathcal{T}$ such that $\pos\bR\subseteq \bigcup_{T\in\mathcal{T}} T$, for each $T\in\mathcal{T}$ the generators can be chosen to be elements in $S$, and any pair $T$,$T'\in \mathcal{T}$ intersect in a common face of $T$ and $T'$. Similarly as before, we call $\mathcal{T}$ a \emph{unimodular triangulation of $\pos\bR$ using $S$} if each $T\in\mathcal{T}$ is unimodular. The important connection between unimodular covers/triangulations and the integer Carath\'eodory rank is given by the following observation: Let $C$ be pointed rational cone and $S\subseteq \mathcal{H}(C)$. If there exists a unimodular cover or even a unimodular triangulation of $C$ using $S$, then this implies that $C$ has the (ICP). This is true because unimodular cones have the (ICP) and every integral vector in $C$ is contained in a unimodular cone spanned by the Hilbert basis elements since $S\subseteq \mathcal{H}(C)$. It is known that having a unimodular triangulation given by the Hilbert basis elements is a strictly stronger property than having a unimodular cover given by the Hilbert basis elements, meaning there are cones that have a suitable unimodular cover but admit no unimodular triangulation \citep{bouvierspringberg1992triangulationstrongerthancover,firlaziegler99hilberttriangandcover}. Similarly, having a unimodular cover given by the Hilbert basis elements is a strictly stronger property than having the (ICP) \citep{bruns2007integralcara}. 
	For an exposition concerning unimodular covers and triangulations, we refer the reader to \citep[Section 9.3]{deloerarambausantostriangulations2010}.

	\subsection{Lattices}\label{ss_notation_lattices}
	We introduce lattices and discuss several useful properties that we need throughout the paper; see \citep{lekkerkerker2014geometry} for an introduction to lattices. A \textit{lattice} $\Lambda$ is a discrete subgroup of $\R^n$. A basis of $\Lambda$ consists of $k$ linearly independent vectors $\bb^1,\ldots,\bb^k\in \Lambda$ such that $\Lambda = (\bb^1,\ldots,\bb^k)\Z^k$. If $k = n$, the lattice is full-dimensional. The determinant of $\Lambda$ is given by $\det\Lambda = \sqrt{\det\bB^T\bB}$ for $\bB = (\bb^1,\ldots,\bb^k)$. Let $\tilde{\Lambda}\subseteq\Lambda$ be a \textit{sublattice}, then the quotient group $\Lambda / \tilde{\Lambda}$ is finite and the cardinality of $\Lambda / \tilde{\Lambda}$ equals $\frac{\det \tilde{\Lambda}}{\det\Lambda}$. 
	The vectors $(\bb^1)^*,\ldots,(\bb^k)^*\in \lin\Lambda$ form the \textit{dual basis} of $\bb^1,\ldots,\bb^k$ if 
	\begin{align*}
		(\bb^i)^T(\bb^j)^* = \begin{cases}
			1, \text{ for } i = j,\\
			0, \text{ otherwise }
		\end{cases}.
	\end{align*}
	The set
	\begin{align*}
		\Lambda^* := \lbrace \bx\in \lin \Lambda : \by^T\bx\in \Z \text{ for all }\by\in\Lambda\rbrace
	\end{align*}
	is a lattice and referred to as the \emph{dual lattice of $\Lambda$}. 
	Note that $(\bb^1)^*,\ldots,(\bb^k)^*\in \Lambda^*$ is a basis of $\Lambda^*$ if $\Lambda = (\bb^1,\ldots,\bb^k)\Z^k$ for $i\in\lbrack k\rbrack$. Moreover, we have $\det \Lambda^* = (\det \Lambda)^{-1}$. In the case $k = n$, the dual basis of $\bB= (\bb^1,\ldots,\bb^n)$ is given by the columns of $\bB^{-T}$. 
	
	\subsection{Projections and the integer Carath\'{e}odory rank}\label{ss_notation_proj}
	Let $M\subseteq \R^n$ be a linear subspace. For a set $X\subseteq \R^n$, we denote by $X|M$ the orthogonal projection of $X$ onto $M$. If $X = \lbrace \bx\rbrace$, we simply write $\bx|M$ and treat $\bx|M$ as a vector. We denote by $\br^\perp$ the orthogonal complement of the linear space spanned by the vector $\br\in \R^n\backslash\lbrace\bm{0}\rbrace$. 
	Let $\br^1,\ldots,\br^k$ be linearly independent with $\bR= (\br^1,\ldots,\br^k)$ and $M=(\br^k)^\perp$. Observe that $\pos\bR | M$ is a $(k-1)$-dimensional pointed cone given by the generators $\br^1|M,\ldots,\br^{k-1}|M$; see also the first part of Lemma \ref{lemma_projection_property}. 
	We define $\Cr(\pos\bR | M)$ to be the integer Carath\'{e}odory rank of $\pos\bR | M$ with respect to the projected lattice $(\lin\bR\cap\Z^n)| M$. Indeed, this is equivalent to our setting, where we consider the integer lattice, due to the following transformation: If $\bv^1,\ldots,\bv^{k-1}$ is a basis of $(\lin \bR \cap \Z^n)|M$, we can supplement this basis with $\bw^1,\ldots,\bw^{n-k+1}$ such that $\bB = (\bv^1,\ldots,\bv^{k-1},\bw^1,\ldots,\bw^{n-k+1})$ is invertible. We return to the integer lattice by transforming everything with $\bB^{-1}$. So we obtain the cone $\bB^{-1}\cdot( \pos \bR | M)$ with $\bB^{-1}\cdot(\br^i|M)\in\Z^n$ for all $i\in\lbrack k - 1\rbrack$ and the lattice $\bB^{-1}\cdot((\lin\bR\cap\Z^n)|M)\subseteq\Z^n$. This does not alter the integer Carath\'{e}odory rank. 
	
	\section{A projective lemma}\label{sec_projective_lemma}
	We formulate the key result for proving Theorems \ref{thm_simplicial_icp} and \ref{thm_simplicial_icr}.
	\begin{lemma}\label{lemma_bounded_complexity}
		Let $\br^1,\ldots,\br^k\in\Z^n$ be linearly independent, $\bR=(\br^1,\ldots,\br^k)$, and $L = \lin\bR$.
		\begin{enumerate}
			\item If $(\br^i)^* \in (L\cap\Z^n)^*$ for some $i\in\lbrack k\rbrack$ , then 
			\begin{align*}
				\Cr(\pos\bR)\leq \Cr(\pos\lbrace\br^1,\ldots,\br^{i-1},\br^{i + 1},\ldots,\br^k\rbrace) + 1.
 			\end{align*}
			\item If $(\br^i)^*- (\br^j)^*\in (L\cap\Z^n)^*$ for some $i,j\in\lbrack k\rbrack$ with $i\neq j$, then
			\begin{align*}
				\Cr(\pos\bR)\leq\max\left\lbrace \Cr(\pos\bR|(\br^i)^\perp), \Cr(\pos\bR|(\br^j)^\perp)\right\rbrace + 1.
			\end{align*}
		\end{enumerate}
	\end{lemma}
	The aim of this section is to prove Lemma \ref{lemma_bounded_complexity}. To do so, we need two results, which we state and prove below.
	
	\begin{lemma}
		\label{lemma_coeffdiff_zero}
		Let $\br^1,\ldots,\br^k\in \Z^n$ be linearly independent, $\bR=(\br^1,\ldots,\br^k)$, and $L = \lin \bR$. Then $(\br^i)^* - (\br^j)^*\in (L\cap\Z^n)^*$ for $i,j\in\lbrack k\rbrack$ if and only if $\lambda_i - \lambda_j = 0$ for all $\by\in \parr\bR$ with $\by = \sum_{l=1}^k \lambda_l\br^l$.
	\end{lemma}
	\begin{proof}
		Let $(\br^i)^* - (\br^j)^*\in (L\cap\Z^n)^*$. We have
		\begin{align*}
			\lambda_i - \lambda_j = ((\br^i)^* - (\br^j)^*)^T\by\in\Z
		\end{align*}
		for $\by \in \parr\bR$ with $\by = \sum_{l=1}^k \lambda_l\br^l$ by the definition of the dual lattice. As $\lambda_i,\lambda_j\in \lbrack 0,1 )$, we obtain $\lambda_i - \lambda_j = 0$.
		
		We prove the converse direction. At first, we claim that the Hilbert basis of $\pos \bR$ is contained in the parallelepiped spanned by $\br^1,\ldots,\br^k$.
		Let $\by\in \pos\bR\cap\Z^n$ be such that $\by=\sum_{i=1}^k\lambda_i\br^i$ and $\lambda_j> 1$ for some $j\in\lbrack k\rbrack$, i.e., $\by$ is not contained in the parallelepiped spanned by $\br^1,\ldots,\br^k$. Then we can decompose $\by$ into two integral vectors such that
		\begin{align*}
			\by = \left((\lambda_j - 1)\br^j + \sum_{i = 1, i\neq j}^k \lambda_i\br^i \right) + \br^j.
		\end{align*}
		Both vectors are integral, contained in $\pos\bR$, and not $\bm{0}$. Therefore, $\by$ is not a Hilbert basis element. 
		Thus, we get that the Hilbert basis is contained in $\parr\bR\cup\lbrace \br^1,\ldots,\br^k\rbrace$. So the elements in $\parr\bR$ combined with $\br^1,\ldots,\br^k$ generate $\pos\bR \cap \Z^n$ with integral and non-negative coefficients. We claim that they integrally generate $L\cap\Z^n$ as well. 
		To see this, take some $\bw \in L\cap\Z^n$. We have
		\begin{align*}
			\bw = \sum_{i=1}^k \mu_i\br^i = \sum_{i=1}^k \lfloor\mu_i\rfloor \br^i + \sum_{i=1}^k (\mu_i - \lfloor \mu_i\rfloor)\br^i,
		\end{align*}
		where the first term is an integer vector and, thus, the second term is contained in $\pos\bR \cap \Z^n$. Hence, the second term is a (non-negative) integral combination of the Hilbert basis. We obtain that the set $\parr\bR\cup\lbrace \br^1,\ldots,\br^k\rbrace$ is a generating set of the lattice $L\cap \Z^n$. 
		Since $((\br^i)^* - (\br^j)^*)^T\by = 0 \in \Z$ for all $\by\in\parr\bR$ and $((\br^i)^* - (\br^j)^*)^T\br^l\in\lbrace 0,\pm 1\rbrace$ for $l\in\lbrack k\rbrack$, we conclude $(\br^i)^* - (\br^j)^*\in (L\cap\Z^n)^*$.
	\end{proof}
	
	The proof of Lemma \ref{lemma_bounded_complexity} consists of a projection argument. In the following, we verify that $\pos\bR$ and $\parr\bR$ behave well under certain orthogonal projections.
	\begin{lemma}
		\label{lemma_projection_property}
		Let $\br^1,\ldots,\br^k\in \Z^n$ be linearly independent, $\bR = (\br^1,\ldots,\br^k)$, and $L=\lin\bR$. Then we have
		\begin{enumerate}
			\item $\pos\bR|(\br^i)^\perp = \pos \lbrace \br^1|(\br^i)^\perp,\ldots,\br^{i-1}|(\br^i)^\perp,\br^{i+1}|(\br^i)^\perp,\ldots,\br^k|(\br^i)^\perp\rbrace$ and
			\item $\parr\bR|(\br^i)^\perp = \parr\lbrace\br^1|(\br^i)^\perp,\ldots,\br^{i-1}|(\br^i)^\perp,\br^{i+1}|(\br^i)^\perp,\ldots,\br^k|(\br^i)^\perp\rbrace$, where the second set is defined with respect to $(L\cap\Z^n)|(\br^i)^\perp$.
		\end{enumerate}
	\end{lemma}
	
	\begin{proof}
		The linearity of the orthogonal projection onto $(\br^i)^\perp$ and the fact that $\br^i | (\br^i)^\perp = \bm{0}$ imply the first claim and the inclusion $"\subseteq"$ in the second claim.
		
		Without loss of generality let $i = 1$. Given $\by\in\parr\lbrace\br^2|(\br^1)^\perp,\ldots,\br^k|(\br^1)^\perp\rbrace$, there exist $\lambda_2,\ldots,\lambda_k\in \lbrack 0,1)$ such that $\by = \sum_{j = 2}^k\lambda_j(\br^j|(\br^1)^\perp)$. 
		We also know there exists $\lambda \in \R$ such that $\sum_{j = 2}^k\lambda_j\br^j + \lambda \br^1\in\Z^n$ since $\by\in(L\cap\Z^n)|(\br^1)^\perp$. Subtracting $\lfloor \lambda\rfloor\br^1\in\Z^n$ yields
		\begin{align*}
			\tilde{\by}=\sum_{j = 2}^k\lambda_j\br^j + (\lambda-\lfloor\lambda\rfloor) \br^1\in\parr\bR.
		\end{align*}
		The claim follows from the observation that $\tilde{\by}|(\br^1)^\perp = \by$.
	\end{proof}
	We are in the position to prove the main statement of this section.
	\begin{proof}[Proof of Lemma \ref{lemma_bounded_complexity}]
		We select an arbitrary $\bz\in \pos\bR \cap \Z^n$ with $\bz = \sum_{l=1}^k \mu_l\br^l$ for $\mu_1,\ldots,\mu_k\in \R_{\geq 0}$. 
		For the first statement, we observe that
		\begin{align*}
			\mu_i =  \bz^T(\br^i)^*\in \Z
		\end{align*}
		as $(\br^i)^*\in (L\cap\Z^n)^*$. Hence, $\bz - \mu_i\br^i \in \pos\lbrace\br^1,\ldots,\br^{i-1},\br^{i+1},\ldots,\br^k\rbrace$. This implies already the first statement.
		
		For the second statement, we assume without loss of generality that $\mu_i\geq \mu_j$. There are $\by^1|(\br^i)^\perp,\ldots,\by^s|(\br^i)^\perp$ and coefficients $\sigma_1,\ldots,\sigma_{s}\in \N$ such that 
		\begin{align*}
			\bz|(\br^i)^\perp = \sum_{l = 1}^{s} \sigma_l(\by^l | (\br^i)^\perp)
		\end{align*}
		with $s\leq \Cr(\pos\bR|(\br^i)^\perp)$ and $\by^1,\ldots,\by^s\in\parr\bR\cup\lbrace \br^1,\ldots,\br^k\rbrace$ by Lemma \ref{lemma_projection_property}. 
		Hence, there exists some $\sigma\in \R$ that satisfies
		\begin{align*}
			\bz = \sigma\br^i + \sum_{l = 1}^{s} \sigma_l\by^l.
		\end{align*}
		Our goal is to show that $\sigma\in\N$, which then implies the second statement. 
		There are $\tau_1^t,\ldots,\tau_k^t\in\lbrack 0,1\rbrack$ such that $\by^t = \sum_{l=1}^k\tau_l^t\br^l$ for each $t\in\lbrack s\rbrack$. So we get
		\begin{align*}
			\bz = \bR\left(\sum_{l = 1}^{s}\sigma_l\btau^t + \sigma \be^i\right),
		\end{align*} 
		where $\btau^t\in\R^n$ with $\btau^t_j = \tau^t_j$ for all $j\in\lbrack n\rbrack$ and $t\in\lbrack s\rbrack$. This yields
		\begin{align*}
			0\leq\mu_i - \mu_j = \sigma + \sum_{l = 1}^s \sigma_l\tau^l_i - \sum_{l = 1}^s \sigma_l\tau^l_j.
		\end{align*}
		Further, we obtain $\tau^t_i = \tau^t_j$ if $\by^t\in \parr\bR$ from Lemma \ref{lemma_coeffdiff_zero}. This implies that the difference of the two sums above is $0$ provided that $\by^t\in \parr\bR\cup\lbrace \br^1,\ldots\br^k\rbrace \backslash \lbrace \br^i,\br^j\rbrace$ for all $t \in \lbrack s\rbrack$. In this special case, we get $\sigma\in \N$ as 
		\begin{align*}
			\sigma = \mu_i - \mu_j = ((\br^i)^* - (\br^j)^*)^T\bz \in \Z
		\end{align*}
		and $\mu_i-\mu_j\geq 0$. 
		So we are left with the case $\by^l = \br^i$ or $\by^l = \br^j$ for some $l \in \lbrack s\rbrack$. Note that $\by^l = \br^i$ is not possible since this would mean $\by^l|(\br^i)^\perp = \bm{0}$. Hence, we assume $\by^l = \br^j$. Here, the difference of the two sums above is $-\sigma_l$. As $\sigma_l\in\N$ by construction, we get $\sigma\in \Z$ by the same argument as before. Furthermore, we have $0\leq \mu_i - \mu_j\leq \sigma$. We conclude that $\sigma\in\N$, which proves the second part of the lemma.
	\end{proof}
	
	\section{Proofs of Theorems \ref{thm_simplicial_icp} and \ref{thm_simplicial_icr}}
	\label{section_main_proofs}
	To prove Theorem \ref{thm_simplicial_icp}, Theorem \ref{thm_simplicial_icr}, and the Corollary \ref{cor_three_coset} below, we employ the following strategy: First, we observe that we are in the situation of the first or second case of Lemma \ref{lemma_bounded_complexity}. After applying the lemma once, we argue that the multiplicity and coset structure of the elements in $\parr \bR$ behave well. Afterwards, we transform everything back to integer lattice via the transformation outlined in Section \ref{ss_notation_proj} if necessary and repeat this procedure. 
	Before we begin to prove Theorems \ref{thm_simplicial_icp} and \ref{thm_simplicial_icr}, we have to ensure that the second step above is well-defined.
	\begin{lemma}\label{lemma_projection_coset}
		Let $\br^1,\ldots,\br^k$ be linearly independent, $\bR = (\br^1,\ldots,\br^k)$, and $L = \lin\bR$. Then
		\begin{enumerate}
			\item $\Delta(\br^1|(\br^i)^\perp,\ldots,\br^{i-1}|(\br^i)^\perp,\br^{i+1}|(\br^i)^\perp,\ldots,\br^k|(\br^i)^\perp)\leq \Delta(\br^1,\ldots,\br^k)$, where the left term is defined with respect to $(L\cap\Z^n)|(\br^i)^\perp$, and
			\item $(\br^s|(\br^i)^\perp)^*-(\br^t|(\br^i)^\perp)^*\in ((L\cap\Z^n)|(\br^i)^\perp)^*$ if $(\br^s)^*-(\br^t)^*\in (L\cap\Z^n)^*$ for $s,t\in\lbrack k\rbrack$ with $s\neq t$. 
		\end{enumerate}
	\end{lemma}
	\begin{proof}
		We begin by proving the first statement. By Lemma \ref{lemma_projection_property}, we know that $\parr\lbrace\br^1|(\br^i)^\perp,\ldots,\br^{i-1}|(\br^i)^\perp,\br^{i+1}|(\br^i)^\perp,\ldots,\br^k|(\br^i)^\perp\rbrace = \parr\bR|(\br^i)^\perp$. This implies the first statement since 
		\begin{align*}
			\Delta(\br^1|(\br^i)^\perp,\ldots,\br^{i-1}|(\br^i)^\perp,\br^{i+1}|(\br^i)^\perp,\ldots,\br^k|(\br^i)^\perp) &= \left|\parr\bR|(\br^i)^\perp\right| \\ &\leq \left|\parr\bR\right| \\ &=\Delta(\br^1,\ldots,\br^k).
		\end{align*}
		For the second claim, let $\tilde{\by}\in \parr\lbrace\br^1|(\br^i)^\perp,\ldots,\br^{i-1}|(\br^i)^\perp,\br^{i+1}|(\br^i)^\perp,\ldots,\br^k|(\br^i)^\perp\rbrace$ with $\tilde{\by}= \by|(\br^i)^\perp$ for $\by = \sum_{l=1}^k\lambda_l\br^l$. We obtain
		\begin{align*}
			\tilde{\by}^T\left((\br^s|(\br^i)^\perp)^* - (\br^t|(\br^i)^\perp)^*\right) &= (\by|(\br^i)^\perp)^T\left((\br^s|(\br^i)^\perp)^* - (\br^t|(\br^i)^\perp)^*\right) \\ &= \sum_{l=1}^{k}\lambda_l (\br^l|(\br^i)^\perp)^T\left((\br^s|(\br^i)^\perp)^* - (\br^t|(\br^i)^\perp)^*\right) \\ &= \lambda_s - \lambda_t = 0,
		\end{align*}
		where we utilize that $\lambda_s - \lambda_t = 0$ by Lemma \ref{lemma_coeffdiff_zero}.
		This holds for all elements in $\parr\bR|(\br^i)^\perp$. After using the transformation outlined in Section \ref{ss_notation_proj} from $(L\cap\Z^n)|(\br^i)^\perp$ to $\Z^n$, we can apply Lemma \ref{lemma_coeffdiff_zero} as the coefficients remain unchanged. We conclude that $(\br^s|(\br^i)^\perp)^*-(\br^t|(\br^i)^\perp)^*\in ((L\cap\Z^n)|(\br^i)^\perp)^*$.
	\end{proof}
	This in combination with Lemma \ref{lemma_bounded_complexity} and Theorem \ref{thm_sebo_three} has the following direct implication.
	\begin{corollary}\label{cor_three_coset}
		Let $\br^1,\ldots,\br^k\in\Z^n$ be linearly independent, $\bR = (\br^1,\ldots,\br^k)$, and $L=\lin\bR$. Then $\pos \bR$ has the (ICP) if $\lbrace(\br^1)^* + (L\cap\Z^n)^*,\ldots,(\br^k)^* + (L\cap\Z^n)^*\rbrace\subseteq(L\cap\bR\Z^n)^* / (L\cap\Z^n)^*$ contains at most three non-trivial pairwise different elements.
	\end{corollary}
	\begin{proof}
		Every three-dimensional cone has the (ICP) due to Theorem \ref{thm_sebo_three}. Therefore, we assume that $k\geq 4$. The assumption that $\lbrace(\br^1)^* + (L\cap\Z^n)^*,\ldots,(\br^k)^* + (L\cap\Z^n)^*\rbrace$ contains at most three non-trivial pairwise different elements implies that there either exists $i\in\lbrack k\rbrack$ such that $(\br^i)^*\in (L\cap\Z^n)^*$ or there exist pairwise different indices $i,j\in\lbrack k\rbrack$ such that $(\br^i)^*-(\br^j)^*\in (L\cap\Z^n)^*$. Hence, we are in one of the cases of Lemma \ref{lemma_bounded_complexity}. 
		
		In the first case, we pass to $\pos\lbrace\br^1,\ldots,\br^{i-1},\br^{i+1},\ldots\br^k\rbrace$. This is a $(k-1)$-dimensional cone, which still satisfies that $\lbrace(\br^1)^* + (\tilde{L}\cap\Z^n)^*,\ldots,(\br^{i-1})^* + (\tilde{L}\cap\Z^n)^*,(\br^{i + 1})^* + (\tilde{L}\cap\Z^n)^*,\ldots,(\br^k)^* + (\tilde{L}\cap\Z^n)^*\rbrace$ contains at most three non-trivial pairwise different elements for $\tilde{L} = \lin\lbrace\br^1,\ldots,\br^{i-1},\br^{i+1},\ldots\br^k\rbrace$.
		
		In the second case, we know from the second part of Lemma \ref{lemma_projection_coset} that $\lbrace (\br^1|(\br^i)^\perp)^* + \Lambda,\ldots,(\br^{i-1}|(\br^i)^\perp)^* + \Lambda,(\br^{i+1}|(\br^i)^\perp)^* + \Lambda,\ldots,(\br^{k}|(\br^i)^\perp)^* + \Lambda\rbrace$ contains at most three non-trivial pairwise different elements, where we have $\Lambda = ((L\cap\Z^n)|(\br^i)^\perp)^*$. Transforming back to the integer lattice via the transformation outlined in Section \ref{ss_notation_proj} gives us a $(k-1)$-dimensional cone defined by integral generators.
		
		Hence, we can apply Lemma \ref{lemma_bounded_complexity} as long as the dimension is at least four. In each step, we only add one element to our non-zero integral combination. We repeat this procedure until the dimension equals three. In this case, we utilize Theorem \ref{thm_sebo_three} to finish the proof of the statement.
	\end{proof}
	We are in the position to prove Theorem \ref{thm_simplicial_icp} and Theorem \ref{thm_simplicial_icr}. Here, we apply Lemma \ref{lemma_bounded_complexity} and Lemma \ref{lemma_projection_coset} in a similar manner as in the proof of Corollary \ref{cor_three_coset}. The main part of the proof is concerned with constructing a unimodular cover in the case when $\left|\det\bR\right| = 5$. 
	\begin{proof}[Proof of Theorem \ref{thm_simplicial_icp}]
		The assumption $1\leq \left|\det\bR\right|\leq 5$ implies that the number of cosets in $\Z^n / \bR\Z^n$ is bounded by five. Hence, by duality, the number of cosets in $(\bR\Z^n)^* / \Z^n$ is bounded by five as well. Since one of those cosets is trivial, the set $\lbrace (\br^1)^* + \Z^n,\ldots,(\br^n)^* + \Z^n\rbrace$ contains at most four non-trivial pairwise different elements. In particular, if $\left|\det\bR\right|\leq 4$, the number of non-trivial pairwise different cosets is bounded by three. In this case, the claim follows from Corollary \ref{cor_three_coset}.
		
		It is left to show that $\pos\bR$ has the (ICP) if $\left|\det\bR\right| = 5$. Following the arguments given in the proof of Corollary~\ref{cor_three_coset}, we assume that $\lbrace (\br^1)^* + \Z^n,\ldots,(\br^n)^* + \Z^n\rbrace$ does not satisfy one of the two assumptions in Lemma~\ref{lemma_bounded_complexity}. In other words, there exists no trivial element in $\lbrace (\br^1)^* + \Z^n,\ldots,(\br^n)^* + \Z^n\rbrace$ and all elements are pairwise different. Since $\left|\det\bR\right| = 5$, we get that $n = 4$. 
		Moreover, we have $\Z^4 / \bR\Z^4\cong \Z/5\Z$, which implies that $\Z^4 / \bR\Z^4$ is cyclic. So there exists an element $\by^1\in\parr\bR$ such that $\by^1 + \bR\Z^n$ generates $\Z^4 / \bR\Z^4$. We can write $\by^1 = \lambda_1\br^1 + \lambda_2\br^2 + \lambda_3\br^3 + \lambda_4\br^4$ for $\lambda_1,\lambda_2,\lambda_3,\lambda_4\in \frac{1}{5}\lbrace 0,1,2,3,4\rbrace$ by Cramer's rule. Further, we can assume that $\lambda_i\neq\lambda_j$ for $i,j\in\lbrack 4\rbrack$ with $i\neq j$. Otherwise, we have that $\mu_i = \mu_j$ for all $\by\in\parr\bR$ with $\by = \mu_1\br^1+\ldots+\mu_4\br^4$ as $\by^1+\bR\Z^4$ generates $\Z^4 / \bR\Z^4$. However, this contradicts the assumption that the set $\lbrace (\br^1)^* + \Z^4,\ldots,(\br^4)^* + \Z^4\rbrace$ contains only pairwise different elements by Lemma \ref{lemma_coeffdiff_zero}. 	
		In a similar manner, we have $\lambda_i\neq 0$ for all $i\in\lbrack 4\rbrack$. So we assume without loss of generality that $\lambda_i = \frac{i}{5}$ for $i\in\lbrack 4 \rbrack$. Let us denote by $\by^2,\by^3,\by^4$ the other non-trivial elements in $\parr\bR$. Since $\by^1 + \bR\Z^4$ generates $\Z^4 / \bR\Z^4$, we can write
		\begin{align*}
			(\br^1,\br^2,\br^3,\br^4,\by^1,\by^2,\by^3,\by^4) = \bR\begin{pmatrix}
				1 & 0 & 0 & 0 & 1 / 5 & 2 / 5 & 3 / 5 & 4 / 5 \\
				0 & 1 & 0 & 0 & 2 / 5 & 4 / 5 & 1 / 5 & 3 / 5 \\
				0 & 0 & 1 & 0 & 3 / 5 & 1 / 5 & 4 / 5 & 2 / 5 \\
				0 & 0 & 0 & 1 & 4 / 5 & 3 / 5 & 2 / 5 & 1 / 5	
			\end{pmatrix}.
		\end{align*}
		Observe that the vectors $\by^1,\by^2,\by^3,\by^4$ are contained in the affine hyperplane given by $\lbrace \bx\in\R^4 : (1,1,1,1)\bR^{-1}\bx = 2\rbrace$. Hence, each $\by^i$ can not be written as a non-negative integral combination of elements in $\lbrace\br^1,\br^2,\br^3,\br^4,\by^1,\by^2,\by^3,\by^4\rbrace\backslash\lbrace \by^i\rbrace$. We conclude that the Hilbert basis of $\pos\bR$ equals $\lbrace\br^1,\br^2,\br^3,\br^4,\by^1,\by^2,\by^3,\by^4\rbrace$. 
		
		We show that $\pos\bR$ has the (ICP) by constructing a unimodular cover $\mathcal{C}$ of $\pos\bR$ using the Hilbert basis of $\pos\bR$. The resulting unimodular cover implies the (ICP) since every integer vector in the cone is contained in a unimodular subcone spanned by the Hilbert basis. We proceed as follows: First, we present the construction of the cover in detail. In this process, it is important that the interiors of the subcones do not intersect. For the sake of readability and brevity, we do not formally verify this property. 
		This can be done by formulating for each pair of simplicial subcones a linear program with an arbitrary non-zero linear functional, where the constraints are given by the polyhedral description of the two subcones with strict inequalities. The resulting linear program will turn out to be infeasible as the intersection of the interiors are empty. 
		Alternatively, one can utilize suitable software such as polymake or SageMath that is capable of computing the intersection and dimension of polyhedra. 
		Once we have the collection of subcones, we prove that they indeed cover $\pos\bR$ via a volume argument. 
		
		To motivate the unimodular cover, observe that $\pos\lbrace \by^1,\by^2,\by^3,\by^4\rbrace$ is a three-dimensional cone contained in $\R^4\cap\lbrace \bx\in\R^4 : (1,-1,-1,1)\bR^{-1}\bx = 0\rbrace$. For each one-dimensional face of $\pos\lbrace \by^1,\by^2,\by^3,\by^4\rbrace$, there exists a unique three-dimensional face of $\pos\bR$ such that the cone spanned by the one-dimensional and three-dimensional face is unimodular. For instance, select $\by^1$, a generator of a one-dimensional face of $\pos\lbrace \by^1,\by^2,\by^3,\by^4\rbrace$. Then, the unique three-dimensional face of $\pos\bR$ corresponding to $\by^1$ is $\pos\lbrace \br^2,\br^3,\br^4\rbrace$. Repeating this procedure for every one-dimensional face of $\pos\lbrace \by^1,\by^2,\by^3,\by^4\rbrace$ yields
		\begin{align*}
			&C_{(2,3,4),(1)} = \pos\lbrace \br^2,\br^3,\br^4, \by^1\rbrace, & C_{(1,3,4),(3)} = \pos\lbrace \br^1,\br^3,\br^4, \by^3\rbrace, \\
			&C_{(1,2,4),(2)} = \pos\lbrace \br^1,\br^2,\br^4, \by^2\rbrace, & C_{(1,2,3),(4)} = \pos\lbrace \br^1,\br^2,\br^3, \by^4\rbrace.
		\end{align*}
		By computing determinants, one can verify that each cone is indeed a unimodular subcone of $\pos\bR$. Moreover, it is possible to check that the interiors of these subcones do not intersect. 
		In a similar manner, we fix two-dimensional faces of $\pos\lbrace \by^1,\by^2,\by^3,\by^4\rbrace$. As earlier, for each of those faces, there exists a unique two-dimensional face of $\pos\bR$ such that the resulting subcone is unimodular. This gives us the subcones 
		\begin{align*}
			& C_{(1,2),(2,4)} = \pos\lbrace\br^1,\br^2,\by^2,\by^4\rbrace, & C_{(1,3),(3,4)} = \pos\lbrace \br^1,\br^3,\by^3,\by^4\rbrace, \\
			& C_{(2,4),(1,2)} = \pos\lbrace \br^2,\br^4,\by^1,\by^2\rbrace, & C_{(3,4),(1,3)} = \pos\lbrace \br^3,\br^4,\by^1,\by^3\rbrace.
		\end{align*}
		Again, one can verify that the cones are unimodular and the interiors of all our current subcones do not intersect. 
		We are left with constructing a unimodular cover for the three-dimensional cone $\pos\lbrace \by^1,\by^2,\by^3,\by^4\rbrace$. In addition to this, we still need to utilize the remaining two-dimensional faces of $\pos\bR$, which are $\pos\lbrace \br^1,\br^4\rbrace$ and $\pos\lbrace \br^2,\br^3\rbrace$. We proceed with constructing a unimodular cover for the subcones $\pos\lbrace \br^2,\br^3,\by^1,\by^2,\by^3,\by^4\rbrace$ and $\pos\lbrace \br^1,\br^4,\by^1,\by^2,\by^3,\by^4\rbrace$. This suffices to cover $\pos\bR$, which we verify at the end of the proof.  
		By symmetry, we only present a unimodular cover of $\pos\lbrace \br^2,\br^3,\by^1,\by^2,\by^3,\by^4\rbrace$ in detail. The construction of a unimodular cover for $\pos\lbrace \br^1,\br^4,\by^1,\by^2,\by^3,\by^4\rbrace$ works analogously. 
		We begin with constructing a unimodular triangulation for $\pos\lbrace \by^1,\by^2,\by^3,\by^4\rbrace$ using the two-dimensional subcone $\pos\lbrace\by^2,\by^3\rbrace$. This yields the unimodular subcones
		\begin{align*}
			&C_{(2),(2,3,4)} = \pos\lbrace \br^2,\by^2,\by^3,\by^4\rbrace, & C_{(2,3),(2,3)} = \pos\lbrace\br^2,\br^3,\by^2,\by^3\rbrace, \\& C_{(3),(1,2,3)} = \pos\lbrace \br^3,\by^1,\by^2,\by^3\rbrace. &
		\end{align*}
		We are left with the unimodular subcones 
		\begin{align*}
			& C_{(2,3),(1,2)} = \pos\lbrace\br^2,\br^3,\by^1,\by^2\rbrace, & C_{(2,3),(3,4)} = \pos\lbrace\br^2,\br^3,\by^3,\by^4\rbrace.
		\end{align*}
		Again, one can verify that the interiors have an empty intersection. When constructing a unimodular cover for the other subcone $\pos\lbrace \br^1,\br^4,\by^1,\by^2,\by^3,\by^4\rbrace$, one has to choose the two-dimensional subcone $\pos\lbrace\by^1,\by^4\rbrace$ to obtain a unimodular triangulation of $\pos\lbrace \by^1,\by^2,\by^3,\by^4\rbrace$ and select the corresponding subcones as above. 
		
		Let $\mathcal{C}$ be the collection of the subcones above, including the unmentioned subcones in the unimodular cover of $\pos\lbrace \br^1,\br^4,\by^1,\by^2,\by^3,\by^4\rbrace$. We claim that $\mathcal{C}$ is a unimodular cover $\pos\bR$. To show this, let $H=\lbrace\bx\in\R^4 : (1,1,1,1)\bR^{-1}\bx \leq 2\rbrace$. Since it is possible to verfiy that the pairwise intersection of each subcone in $\mathcal{C}$ has empty interior, it suffices to argue that 
		\begin{align*}
			\vol\left(\bigcup_{C\in\mathcal{C}} C\cap H\right) \geq \vol(\pos\bR\cap H),
		\end{align*}
		where $\vol(\cdot)$ denotes the four-dimensional Lebesgue measure. This is indeed sufficient because the volume inequality implies 
		$\pos\bR\cap H\subseteq\bigcup_{C\in\mathcal{C}} C\cap H$. By scaling, this gives us $\pos\bR \subseteq \bigcup_{C\in\mathcal{C}} C$. In the following, we prove the volume inequality above. Observe that we have $\pos\bR\cap H = 2\cdot \conv\lbrace \bm{0},\br^1,\br^2,\br^3,\br^4\rbrace$, where $\conv X$ denotes the convex hull of the set $X\subseteq\R^n$. So we deduce 
		\begin{align*}
			\vol(\pos\bR\cap H) = \frac{2^4}{4!}\left|\det\bR\right| = \frac{10}{3}.
		\end{align*}
		To determine the volume of the simplices corresponding to unimodular subcones in $\mathcal{C}$, we have to take into account that the vertices coming from the generators are scaled to $2\br^1,2\br^2,2\br^3,2\br^4$. So we begin by counting the number of generators, $\br^1,\br^2,\br^3,\br^4$, in each subcone. There are four subcones with precisely three generators, ten subcones with two generators, and four subcones with only one generator. As every subcone is full-dimensional and unimodular, we calculate 
		\begin{align*}
			\vol\left(\bigcup_{C\in\mathcal{C}} C\cap H\right) = \sum_{C\in\mathcal{C}}\vol(C\cap H) = \frac{1}{4!} (4 \cdot 2^3 + 10 \cdot 2^2 + 4 \cdot 2^1) = \frac{80}{4!} = \frac{10}{3}.
		\end{align*}
		So we conclude that $\pos\bR = \bigcup_{C\in\mathcal{C}}C$.
	\end{proof}
	Although the collection of subcones $\mathcal{C}$ from the previous proof form a unimodular cover and their interiors do not intersect, $\mathcal{C}$ does not give a unimodular triangulation of the cone $\pos\bR$. The reason for this is that the three-dimensional cone $\pos\lbrace \by^1,\by^2,\by^3,\by^4\rbrace$ is triangulated by once using $\pos\lbrace \by^2,\by^3\rbrace$ and once using $\pos\lbrace \by^1,\by^4\rbrace$. So there are cones in $\mathcal{C}$ whose intersections do not correspond to a common face.
	
	\begin{proof}[Proof of Theorem \ref{thm_simplicial_icr}]
		Since $n\geq \left|\det\bR\right|$ and $(\bR\Z^n)^* / \Z^n$ has cardinality $\left|\det\bR\right|$, either one of the vectors $(\br^1)^*,\ldots,(\br^n)^*$ is integral or the difference of two pairwise different vectors is integral by the pigeonhole principle. So we are in the situation of Lemma \ref{lemma_bounded_complexity}. 
		Thus, we can reduce the dimension by one. By the first part of Lemma \ref{lemma_projection_coset}, the corresponding lower-dimensional volume of the respective parallelepiped does not exceed $\Delta(\br^1,\ldots,\br^n) = \left|\det\bR\right|$. We repeat this procedure until the dimension equals $\left|\det \bR\right| - 1$. As $\left|\det \bR\right| - 1\geq 5>2$, we can apply Seb\H{o}'s bound, Theorem \ref{thm_sebo_icr}, and obtain
		\begin{align*}
			\Cr(\pos\bR)\leq 2(\left|\det\bR\right| - 1) - 2 + (n - (\left|\det \bR\right| - 1)) = n + \left|\det \bR\right| - 3.
		\end{align*}
	\end{proof}
	
	\section{Special cones with the (ICP)}\label{sec_special_cones}
	Utilizing our methods, we collect special instances of well-known simplicial cones that have the (ICP). 
	In fact, it should not be too challenging to obtain more examples using the methods from the earlier chapters. 
	We proceed by applying the slightly more abstract statement given in Corollary \ref{cor_three_coset}. We will see that the restriction on the number of cosets naturally translates to restrictions on the representation of our following cones. 
	
	The first example is given by $n-1$ standard unit vectors and one additional integral vector $\br$. This setting has been studied with respect to the integer Carath\'{e}odory rank and only special instances with the (ICP) are known; see \citep{henkweis02diophantineandintcara} for some of them. We assume $\br_n = \Delta$ and that the first $n-1$ coordinates of $\br$ admit at most two different values apart from $0$ and $\Delta - 1$, e.g., $\br = (0,\ldots,0,a,\ldots,a,b,\ldots,b,\Delta - 1,\ldots,\Delta -1,\Delta)^T$ for some $a,b\in\lbrace 1,\ldots,\Delta - 2\rbrace$.
	
	\begin{proposition}
		\label{prop_skew_vector}
		Let $\bR = (\be^1,\ldots,\be^{n-1},\br)\in \Z^{n\times n}$ with $\br_n = \Delta \geq 1$. Further, let $I\subseteq\lbrace 0,\ldots,\Delta - 1\rbrace$ such that $\br_i\in I$ for all $i \in \lbrack n - 1\rbrack$. If $\left| I\backslash\lbrace 0,\Delta - 1\rbrace \right|\leq 2$, then
		\begin{align*}
			\Cr(\pos\bR) = n.
		\end{align*}
	\end{proposition}
	\begin{proof}
		We claim that $\lbrace \bR^{-T}\be^1 + \Z^n,\ldots,\bR^{-T} \be^n + \Z^n\rbrace$ contains at most three non-trivial cosets in $\bR^{-T}\Z^n / \Z^n$. Then, the statement follows from Corollary \ref{cor_three_coset}. 
		Observe that $\bR^{-T} = (\be^1,\ldots,\be^{n - 1}, \br^*)^T$, where $\br_i^* = -\frac{\br_i}{\Delta}$ for $i \in \lbrack n-1\rbrack$ and $\br_n^* = \frac{1}{\Delta}$. So we have $\bR^{-T}\be^i\in\Z^n$ if and only if $\br_i = 0$. Further, if $0 \neq \br_k = \br_l$ for $k,l\in\lbrack n - 1\rbrack$, we obtain $\bR^{-T}(\be^k - \be^l)\in\Z^n$. Also, we have $\bR^{-T}(\be^j - \be^n)\in\Z^n$ for $j\in\lbrack n - 1\rbrack$ if $\br_j = \Delta - 1$. Therefore, every non-zero element in $I$ contributes exactly one non-trivial coset. There are at most three non-zero elements in $I$ since $|I\backslash\lbrace 0\rbrace| \leq 1 + |I\backslash \lbrace 0,\Delta - 1\rbrace|\leq 3$. The claim follows from Corollary \ref{cor_three_coset}.
	\end{proof}
	
	For our second example, we assume that there exists $\by \in \intt (\pos\bR)\cap \parr\bR$ with $\by + (\pos\bR \cap \Z^n) = \intt (\pos\bR)\cap \Z^n$. Special cones satisfying this premise are dual cones of simplicial Gorenstein cones. The general class of Gorenstein cones and polytopes plays an important role in the study of toric varieties; see, e.g., \citep{batyrevreflexgorenstein1993} for the first appearance of Gorenstein cones and their related polyhedra.
	
	\begin{proposition}
		\label{prop_gorenstein_4}
		Let $\bR = (\br^1,\ldots,\br^n)\in\Z^{n \times n}$ be such that there exists a vector $\by \in \intt (\pos\bR)\cap \parr\bR$ with $\by + (\pos\bR \cap \Z^n) = \intt (\pos\bR)\cap \Z^n$. Further, let $\left| \det \bR\right|$ have at most four divisors and let $\Z^n / \bR\Z^n$ be cyclic. Then
		\begin{align*}
			\Cr(\pos\bR) = n.
		\end{align*}		
	\end{proposition}
	\begin{proof}
		Let $\by = \bR \blambda$ with $\blambda\in\frac{1}{\left| \det \bR\right|}\lbrace 1,\ldots,\left| \det \bR\right| - 1\rbrace^n$ by Cramer's rule. Observe that $\blambda_i\neq 0$ for $i\in\lbrack n\rbrack$ since $\by$ is in the interior of the cone. As a first step, we prove that the coefficients of $\blambda$ have to be small. To do so, we apply the transformation $\bR^{-1}$ and analyze the elements in $\R^n_{\geq 0}\cap \Lambda$ for $\Lambda = \bR^{-1}\Z^n$. Note that $\bR^{-1}\cdot\parr\bR = \lbrack 0,1)^n\cap\Lambda$. Let $\bmu \in \R^n_{\geq 0}\cap\Lambda$ and let $\bmu_k\neq 0$. If $\bmu$ lies in the interior of $\R^n_{\geq 0}$, we must have $\blambda_k\leq\bmu_k$ by $\blambda + (\R^n_{\geq 0} \cap \Lambda) = \intt \R^n_{\geq 0}\cap \Lambda$. Thus, we assume that $\bmu$ is contained in the boundary of $\R^n_{\geq 0}$, i.e., there exists $l\in\lbrack n\rbrack$ such that $\bmu_l = 0$. Let $I = \lbrace l\in\lbrack n\rbrack : \bmu_l = 0\rbrace$. This yields 
		\begin{align*}
			\bmu + \sum_{l\in I}\be^l\in\intt \R^n_{\geq 0} \cap \Lambda.
		\end{align*}
		As $k\notin I$ and $\blambda + (\R^n_{\geq 0} \cap \Lambda) = \intt \R^n_{\geq 0}\cap \Lambda$, we get $\blambda_k\leq\left(\bmu_k + \sum_{l\in I}\be_k^l\right)=\bmu_k$. As a result, we conclude 
		\begin{align}
			\label{eq_gorenstein_lambda}
			\blambda_k = \min \lbrace \bmu_k : \bmu\in \R^n_{\geq 0} \cap \Lambda \text{ with }\bmu_k \neq 0\rbrace
		\end{align}
		for all $k\in\lbrack n\rbrack$.
		
		As a next step, we show that $\blambda$ has at most three different entries. To do so, we need that $\Z^n / \bR\Z^n$ is cyclic. This is equivalent to $\Lambda / \Z^n$ being cyclic. We want to show that $\blambda + \Z^n$ generates $\Lambda / \Z^n$. Since $\Lambda / \Z^n$ is cyclic, there exists a $\bmu \in \lbrack 0,1)^n \cap \Lambda$ such that $\bmu + \Z^n$ generates $\Lambda / \Z^n$. This means that some entry of $\bmu$ indexed by $j\in\lbrack n\rbrack$ satisfies $\bmu_j = \frac{s}{\left|\det\bR\right|}$ for $s\in\lbrace 1,\ldots,\left|\det\bR\right|\rbrace$, where $s$ and $\left|\det\bR\right|$ are coprime. There exists $m\in\lbrace 1,\ldots,\left|\det\bR\right| - 1\rbrace$ with $ms \equiv 1 \mod \left|\det\bR\right|$, which implies that $m\bmu - \lfloor m\bmu\rfloor\in\lbrack 0,1)^n\cap\Lambda$ has $j$-th entry $\frac{1}{\left|\det\bR\right|}$. So we obtain $\blambda_j = \frac{1}{\left|\det\bR\right|}$ by (\ref{eq_gorenstein_lambda}). This yields that $\blambda + \Z^n$ is a generator of $\Lambda / \Z^n$. 
		We argue that $\blambda$ has at most three different entries. Select $i\in\lbrack n \rbrack$. Let $\blambda_i = \frac{p_i}{q_i}$ for $q_i\in\lbrace 1,\ldots,\left|\det\bR\right|\rbrace$ and $p_i\in\lbrace 1,\ldots, q_i - 1\rbrace$. Since $\blambda + \Z^n$ generates $\Lambda / \Z^n$, we have $q_i \bmu_i\in\N$ for all $\bmu\in \lbrack 0,1)^n\cap\Lambda$. By a similar argument as before, there exists $m\in\lbrace 1,\ldots,q_i - 1\rbrace$ such that $m p_i \equiv 1\mod q_i$ and, thus, $m\bmu - \lfloor m\bmu\rfloor\in\lbrack 0,1)^n\cap\Lambda$ with $\bmu_i = \frac{1}{q_i}$. Again, we conclude that $\blambda_i = \frac{1}{q_i}$ by (\ref{eq_gorenstein_lambda}). We know that $q_i$ divides $\left|\det\bR\right|$ by Cramer's rule. Since there are only four divisors of $\left|\det\bR\right|$, where one of them is $1$, we conclude that $\blambda$ has at most three different entries. 	
		
		As a final step, we claim that $\blambda_k = \blambda_l$ for $k\neq l$ implies $\bR^{-T}(\be^k-\be^l)\in\Z^n$. This property combined with the fact that $\blambda$ has at most three different entries and Corollary \ref{cor_three_coset} prove the statement. 
		Every $\bmu\in\lbrack 0,1)^n\cap\Lambda$ can be written as $m\blambda - \lfloor m \blambda\rfloor$ for $m\in\lbrace 0,\ldots,\left|\det \bR\right| - 1\rbrace$ since $\blambda + \Z^n$ generates $\Lambda / \Z^n$. Therefore, we deduce that $\bmu_k = \bmu_l$ for all $\bmu\in \R^n_{\geq 0} \cap \Lambda$ if $\blambda_k = \blambda_l$. The claim that $\bR^{-T}(\be^k-\be^l)\in\Z^n$ follows from Lemma \ref{lemma_coeffdiff_zero}.
	\end{proof}

	In Proposition \ref{prop_gorenstein_4} we assume among other things that $\by\in\parr\bR$. The statement remains valid if we drop this assumption and suppose that $\by\in \intt (\pos\bR)\cap \Z^n$ with $\by + (\pos\bR \cap \Z^n) = \intt (\pos\bR)\cap \Z^n$. In this case, it is possible to deduce that a face of $\pos\bR$ satisfies the premises of Proposition \ref{prop_gorenstein_4}. 	
	Below, we present a cone that meets the assumptions of Proposition \ref{prop_gorenstein_4} and does not satisfy the assumptions of Proposition \ref{prop_skew_vector}. 
	Let $p,q\in\N$ be distinct prime numbers and $k\in\lbrace 1,\ldots,q - 1\rbrace$ and $l\in\lbrace 1,\ldots, p - 1\rbrace$ such that $kp+lq = pq - 1$. The integers $k$ and $l$ always exist. For instance, choose $k\in\lbrace 1,\ldots,q - 1\rbrace$ such that $k p \equiv q - 1\mod q$, which implies $pq - 1 - kp \equiv 0 \mod q$. So we have $pq - 1 - kp = lq$ for some $l\in\lbrace 1,\ldots, p - 1\rbrace$. We define
	\begin{align*}
		\bR = \begin{pmatrix}
			1 & 0 & l & k\\
			0 & 1 & l & k\\
			0 & 0 & p & 0\\
			0 & 0 & 0 & q
		\end{pmatrix}.
	\end{align*}
	Note that $\det\bR = pq$ has four divisors, $1$, $p$, $q$, and $pq$, and observe that 
	\begin{align*}
		\blambda = \begin{pmatrix}
			\frac{1}{pq}\\
			\frac{1}{pq}\\
			\frac{1}{p}\\
			\frac{1}{q}
		\end{pmatrix}
	\end{align*}
	yields $\by = \bR\blambda\in\Z^4$. So $\by + \bR\Z^4$ generates $\Z^4 / \bR\Z^4$, which implies that $\Z^4 / \bR\Z^4$ is cyclic. If there exists some $\bz\in\intt(\pos\bR)\cap\Z^4$ with $\bz = \bR\bmu$ that is not contained in $\by + (\pos\bR \cap \Z^4)$, then there exists some $j\in\lbrack 4\rbrack$ with $0<\bmu_j < \blambda_j$. As $\blambda_1$ and $\blambda_2$ are already minimal, $j$ has to be either three or four. However, $0<\bmu_j < \blambda_j$ implies that $(\bR\bmu)_j\notin \Z$, a contradiction. So we get $\by + (\pos\bR \cap \Z^4) = \intt (\pos\bR)\cap \Z^4$. Thus, $\by$ meets the assumptions of Proposition \ref{prop_gorenstein_4}. The matrix $\bR$ is already in Hermite normal form; see \citep[Chapter 4]{schrijvertheorylinint86} for more about Hermite normal forms and unimodular transformations. As the Hermite normal form is unique, the matrix $\bR$ can not be transformed into a matrix given by three standard unit vectors and an additional vector $\br$ via a unimodular transformation. So $\pos\bR$ does not satisfy the assumptions of Proposition \ref{prop_skew_vector}.
	\newline
	\newline
	\noindent \textbf{Acknowledgments.} The author is grateful to Martin Henk for several helpful discussions on this topic and valuable remarks on the manuscript. The author also thanks the anonymous reviewer for comments that led to an improved version of the manuscript.
	
	\bibliography{references} 

\begin{thebibliography}{24}
\providecommand{\natexlab}[1]{#1}
\providecommand{\url}[1]{\texttt{#1}}
\expandafter\ifx\csname urlstyle\endcsname\relax
  \providecommand{\doi}[1]{doi: #1}\else
  \providecommand{\doi}{doi: \begingroup \urlstyle{rm}\Url}\fi

\bibitem[Aliev et~al.(2017)Aliev, {J. De Loera}, Oertel, and
  O'Neill]{alievdeloesparelindio2017}
I.~Aliev, {J. De Loera}, T.~Oertel, and C.~O'Neill.
\newblock Sparse solutions of linear diophantine equations.
\newblock \emph{SIAM Journal on Applied Algebra and Geometry}, 1:\penalty0
  239--253, 2017.

\bibitem[Aliev et~al.(2018)Aliev, {J. De Loera}, Eisenbrand, Oertel, and
  Weismantel]{alideloeisoerweissupportint2018}
I.~Aliev, {J. De Loera}, F.~Eisenbrand, T.~Oertel, and R.~Weismantel.
\newblock The support of integer optimal solutions.
\newblock \emph{SIAM Journal on Optimization}, 28:\penalty0 2152--2157, 2018.

\bibitem[Aliev et~al.(2022)Aliev, Averkov, de~Loera, and
  Oertel]{alievAverkovLoeraOertel21}
I.~Aliev, G.~Averkov, J.~A. de~Loera, and T.~Oertel.
\newblock Sparse representation of vectors in lattices and semigroups.
\newblock \emph{Mathematical Programming}, 192:\penalty0 519–546, 2022.

\bibitem[Aliev et~al.(2024)Aliev, Henk, Hogan, Kuhlmann, and
  Oertel]{aliev2024new}
I.~Aliev, M.~Henk, M.~Hogan, S.~Kuhlmann, and T.~Oertel.
\newblock New bounds for the integer {C}arath{\'e}odory rank.
\newblock \emph{SIAM Journal on Optimization}, 34\penalty0 (1):\penalty0
  190--200, 2024.

\bibitem[Batyrev(1994)]{batyrevreflexgorenstein1993}
V.~V. Batyrev.
\newblock Dual polyhedra and mirror symmetry for {C}alabi--{Y}au hypersurfaces
  in toric varieties.
\newblock \emph{J. Alg. Geom.}, 3:\penalty0 493--545, 1994.

\bibitem[Bouvier and
  Gonzalez-Springberg(1992)]{bouvierspringberg1992triangulationstrongerthancover}
C.~Bouvier and G.~Gonzalez-Springberg.
\newblock G-desingularisations de varietes toriques.
\newblock \emph{C. R. Acad. Sci. Paris, Ser. I}, 315:\penalty0 817--820, 1992.

\bibitem[Bruns(2007)]{bruns2007integralcara}
W.~Bruns.
\newblock On the integral {C}arath\'eodory property.
\newblock \emph{Experimental Mathematics}, 16\penalty0 (3):\penalty0 359--365,
  2007.

\bibitem[Bruns and Gubeladze(1999)]{brunsgubeladzenormalpoly1999}
W.~Bruns and J.~Gubeladze.
\newblock Normality and covering properties of affine semigroups.
\newblock \emph{Journal f{\"u}r reine und angewandte Mathematik}, 510:\penalty0
  151--178, 1999.

\bibitem[Bruns and Gubeladze(2009)]{brunsgubeladzepolyringsktheory2009}
W.~Bruns and J.~Gubeladze.
\newblock \emph{Polytopes, {R}ings, and {K}-theory}.
\newblock Springer Monographs in Mathematics, Springer-Verlag, 2009.

\bibitem[Bruns and von Thaden(2017)]{brunsthadenunimodtriangbound2017}
W.~Bruns and M.~von Thaden.
\newblock Unimodular triangulations of simplicial cones by short vectors.
\newblock \emph{Journal of Combinatorial Theory, Series A}, 150:\penalty0
  137--151, 2017.

\bibitem[Bruns et~al.(1999)Bruns, Gubeladze, Henk, Martin, and
  Weismantel]{brunsgubehenkcounterexampleintcara99}
W.~Bruns, J.~Gubeladze, M.~Henk, A.~Martin, and R.~Weismantel.
\newblock A counterexample to an integer analogue of {C}arath{\'e}odory's
  theorem.
\newblock \emph{Journal f{\"u}r reine und angewandte Mathematik}, 510:\penalty0
  179--185, 1999.

\bibitem[de~Pina and Soares(2003)]{pinamatroidcararank2003}
J.C. de~Pina and J.~Soares.
\newblock Improved bound for the {C}arath{\'e}odory rank of the bases of a
  matroid.
\newblock \emph{Journal of Combinatorial Theory, Series B}, 88:\penalty0
  323--327, 2003.

\bibitem[Dubey and Liu(2023)]{dubeyliu2023shortproofsparse}
Y.~Dubey and S.~Liu.
\newblock A short proof of tight bounds on the smallest support size of integer
  solutions to linear equations.
\newblock \url{arXiv:2307.08826}, 2023.

\bibitem[Eisenbrand and Shmonin(2006)]{eisenbrandshmonincaratheodorybounds06}
F.~Eisenbrand and G.~Shmonin.
\newblock Carath{\'e}odory bounds for integer cones.
\newblock \emph{Operations Research Letters}, 34:\penalty0 564--568, 2006.

\bibitem[Firla and Ziegler(1999)]{firlaziegler99hilberttriangandcover}
R.T. Firla and G.~M. Ziegler.
\newblock Hilbert bases, unimodular triangulations, and binary covers of
  rational polyhedral cones.
\newblock \emph{Discrete \& Computational Geometry}, 21\penalty0 (2):\penalty0
  205--216, 1999.

\bibitem[Gijswijt and Regts(2012)]{gijswijtipdicp2012}
D.C. Gijswijt and G.~Regts.
\newblock Polyhedra with the {I}nteger {C}arath{\'e}odory property.
\newblock \emph{Journal of Combinatorial Theory, Series B}, 102:\penalty0
  62--70, 2012.

\bibitem[Gubeladze(2023)]{gubeladzesurveynormal2023}
J.~Gubeladze.
\newblock Normal polytopes: between discrete, continuous, and random.
\newblock \emph{Journal of Pure and Applied Algebra}, 227, 2023.

\bibitem[Henk and Weismantel(2002)]{henkweis02diophantineandintcara}
M.~Henk and R.~Weismantel.
\newblock Diophantine approximations and integer points of cones.
\newblock \emph{Combinatorica}, 22:\penalty0 401--408, 2002.

\bibitem[Lekkerkerker and Gruber(1987)]{lekkerkerker2014geometry}
C.~G. Lekkerkerker and P.~M. Gruber.
\newblock \emph{Geometry of numbers}, volume~37 of \emph{North-Holland
  Mathematical Library}.
\newblock North-Holland Publishing Co., Amsterdam, 1987.

\bibitem[Loera et~al.(2010)Loera, Rambau, and
  Santos]{deloerarambausantostriangulations2010}
J.~De Loera, J.~Rambau, and F.~Santos.
\newblock \emph{Triangulations: Structures for Algorithms and Applications}.
\newblock Springer Berlin, Heidelberg, 1 edition, 2010.

\bibitem[Schrijver(1986)]{schrijvertheorylinint86}
A.~Schrijver.
\newblock \emph{Theory of Linear and Integer Programming}.
\newblock Wiley, 1986.

\bibitem[Seb{\H o}(1990)]{sebohilbertbasisdreidim90}
A.~Seb{\H o}.
\newblock {H}ilbert bases, {C}arath{\'e}odory's theorem and combinatorial
  optimization.
\newblock \emph{Proceedings of the 1st Integer Programming and Combinatorial
  Optimization Conference}, pages 431--455, 1990.

\bibitem[van~der Corput(1931)]{vandercorputhilberbasisunique1931}
J.~van~der Corput.
\newblock Konstruktion der {M}inimalbasis f{\"u}r spezielle {D}iophantische
  {S}ysteme von linear-homogenen {G}leichungen und {U}ngleichungen.
\newblock \emph{Proc. Roy. Acad.}, 34:\penalty0 515--523, 1931.

\bibitem[von Thaden(2021)]{thadenunimodtriangbound2021}
M.~von Thaden.
\newblock Polynomial-size vectors are enough for the unimodular triangulation
  of simplicial cones.
\newblock \emph{Journal of Combinatorial Theory, Series A}, 180, 2021.

\end{thebibliography}
	
\end{document}